\documentclass[10pt]{amsart}
\usepackage{amssymb,amsbsy,amsmath,amsfonts,mathrsfs}
\usepackage[all]{xy}

\title[Cohomology of weakly almost periodic group representations]{On the cohomology of weakly almost periodic group representations}

\author{Uri Bader}
\address{Technion, Haifa, Israel}
\email{uri.bader@gmail.com}

\author {Christian Rosendal}
\address{University of Illinois at Chicago, Chicago, USA}
\email{rosendal.math@gmail.com}
\urladdr{http://homepages.math.uic.edu/$~$rosendal}

\author{Roman Sauer}
\address{Karlsruhe Institute of Technology, Karlsruhe, Germany}
\email{roman.sauer@kit.edu}
\urladdr{http://www.math.kit.edu/iag7/}

\newcommand{\norm}[1]{\lVert#1\rVert}
\newcommand{\Norm}[1]{\big\lVert#1\big\rVert}

\newcommand{\triple}[1]{|\!|\!|#1|\!|\!|}

\newcommand {\Z}{\mathbb Z}
\newcommand {\R}{\mathbb R}

\newcommand {\N}{\mathbb N}

\newcommand{\eps}{\epsilon}
\newcommand{\id}{\operatorname{id}}
\newcommand{\til}{\rightarrow}

\newcommand{\Lim}[1]{\mathop{\longrightarrow}\limits_{#1}}

\newcommand {\del}{ \; \big| \;}

\newcommand {\ku} {\mathcal}

\newcommand{\Id}{{\rm Id}}

\newcommand{\cohom}{H_c}
\newcommand{\redcohom}{\bar{H}_c}

\newcommand{\im}{\operatorname{im}}

\newcommand{\scrG}{\mathscr{G}}

\newtheorem{thm}{Theorem}
\newtheorem{cor}[thm]{Corollary}
\newtheorem{lem}[thm]{Lemma}

\theoremstyle{definition}
\newtheorem{defin}[thm]{Definition}
\newtheorem{rem}[thm]{Remark}
\newtheorem{exa}[thm]{Example}

\thanks{Christian Rosendal was partially supported by a grant from the Simons Foundation (Grant \#229959) and also recognizes support from the NSF (Grants  DMS 0901405 and  DMS 1201295)}

\thanks{Roman Sauer acknowledges support by a grant from the DFG (Grant 1661/3-1)}

\begin{document}


\begin{abstract}
	We initiate a study of cohomological aspects of weakly almost periodic group representations on Banach spaces, in particular, isometric representations on reflexive Banach spaces. Using the Ryll-Nardzewski 	fixed point Theorem, we prove a vanishing result for the restriction map (with respect to a subgroup) in the reduced cohomology of weakly periodic representations. Combined with the Alaoglu-Birkhoff decomposition theorem, 
this generalizes and complements theorems on continuous group cohomology by several authors. 
\end{abstract}

\maketitle

\section{Introduction}
In the following, $V$ denotes a Banach space over the real or complex numbers. Let $B(V)$ denote the algebra of bounded linear operators on $V$ and $GL(V)$ denote the subgroup of $B(V)$ consisting of the invertible bounded operators on $V$. The algebra $B(V)$, and thus also $GL(V)$, is naturally equipped with three different topologies, namely, the {\em operator norm topology}, the {\em strong operator topology (sot)} and the {\em weak operator topology (wot)}.
Here, the strong operator topology on $B(V)$ is that induced by the inclusion of $B(V)$ into the product space $(V,\norm\cdot)^V$, while the weak operator topology is that induced from the inclusion into $(V,{\rm weak})^V$. In other words, for a net $T_i\in B(V)$ and $T\in B(V)$, we have
\begin{align*}
T_i\Lim{sot}T&\iff\forall_{v\in V}~\lim_i\norm{T_iv-Tv}=0,\\
T_i\Lim{wot}T&\iff\forall_{v\in V,\psi\in V^*}~\lim_i\psi(T_iv-Tv)=0.
\end{align*}

\begin{defin}Let $G$ be a group.
A linear representation $\rho\colon G\til GL(V)$ is said to be \emph{weakly almost periodic (wap)} if, for all $v\in V$, the orbit $\rho(G)v$ is relatively weakly compact in~$V$. 	
\end{defin}
We note that, if $V$ is reflexive, every norm bounded subset of $V$ is relatively weakly compact, which implies that any isometric linear representation $\rho\colon G\til GL(V)$ is wap. This provides the main source of wap representations, though, as we indicate later, there are several other interesting examples.

A linear representation $\rho\colon G\til GL(V)$ of a group $G$ is said to be {\em bounded} if, for all $v\in V$ and $\psi\in V^*$, we have $\sup_{g\in G}\Norm{\psi\big(\rho(g)v\big)}<\infty$, in which case, by the uniform boundedness principle, we actually have $\sup_{g\in G}\norm {\rho(g)}<\infty$. Note that, if $\rho\colon G\til GL(V)$ is a bounded representation of $G$, then, with respect to the equivalent norm $\triple{v}=\sup_{g\in G}\norm{\rho(g)v}$ on $V$, $\rho$ is an isometric linear representation of $G$.

It follows from the comments above that wap representations are also bounded. Moreover, by applying Tikhonov's Theorem to the subspace $\prod_{v\in V}\rho(G)v$ of $(V,{\rm weak})^V$, one sees that, in this case, $\rho(G)$ is relatively compact in $B(V)$ with respect to the weak operator topology. In other words, the following two conditions are equivalent
\begin{enumerate}
\item $\rho\colon G\til GL(V)$ is wap,
\item $\rho(G)$ is relatively wot-compact in $B(V)$.
\end{enumerate}

In Section~\ref{sec: wap}, we present some background material on weakly almost periodic group representations and the associated decomposition theorems. One of the main results, the Alaoglu-Birkhoff decomposition, is used in our cohomological applications described below, but we also believe that the other mentioned decomposition theorems may be of independent interest in connection with cohomological aspects of group representations.

If $\rho\colon G\to GL(V)$ is a linear representation, we shall, most often, 
simply eliminate the reference to $\rho$ and say that $V$ is a \emph{Banach $G$-module}. 
Adjectives describing the various topological properties of $\rho$ will be used for the Banach $G$-module as well. In particular, a Banach $G$-module $V$ will be said to be {\em continuous} if the corresponding linear representation is sot-continuous, i.e., if for every $v\in V$, the map $g\in G\mapsto \rho(g)v \in V$ is continuous with respect to the norm on $V$. Also, we let $V^G=\{v\in V\;|\: \rho(g)v=v \text{ for all }g\in G\}$ denote the subspace of invariant vectors.

Let us now turn to cohomology. Theorem \ref{thm: restriction to normal subgroup} below is our main result in this context. As shown in 
Section~\ref{sec:proofs_of_cohomological_statements}, it follows 
rather easily from the Ryll-Nardzewski fixed point Theorem and general 
cohomological techniques. It is a generalization (to higher degrees, non-unitary coefficients and general topological groups) of an important ingredient 
in Y.~Shalom's cohomological rigidity result for unitary representations of locally compact groups~\cite{shalom}. 

\begin{thm}\label{thm: restriction to normal subgroup}
Let $G$ be a topological group and $V$ be a continuous wap Banach $G$-module. Assume that  $N$ and $C$ are subgroups of $G$ with $C$ lying in the centralizer of $N$. If $V^C=\{0\}$, i.e., $C$ fixes no non-zero vector in $V$, then the restriction map 
	\[
		\redcohom^n(G, V)\to \redcohom^n(N, V)
	\]
    in reduced continuous cohomology is zero for every $n\geqslant 0$. 

In particular, if $\alpha\colon G\curvearrowright V$ is a continuous affine isometric action of $G$ on $V$ with weakly almost periodic linear part $\rho$ so that $\rho(C)$ has no non-zero fixed vectors on $V$, then $\alpha\colon N\curvearrowright V$ almost fixes a point on $V$. 
\end{thm}


As a rather immediate corollary of Theorem \ref{thm: restriction to normal subgroup}, we have the following result extending Shalom's Theorem 3.1 \cite{shalom} from locally compact to arbitrary topological groups. Unfortunately, we do not know if this also holds for arbitrary wap or even uniformly convex coefficients.

\begin{cor}\label{thm: shalom}
	Let $\rho \colon G_1\times G_2\til \ku U(\ku H)$ be a continuous unitary representation of a direct product of topological groups.  Then
	\[
		\bar H_c^1(G_1\times G_2, \ku H)\cong \bar H_c^1(G_1, \ku H^{G_2})\oplus \bar H_c^1(G_2, \ku H^{G_1}).
	\]
In particular, if $\ku H^{G_1}=\ku H^{G_2}=\{0\}$ and $\alpha\colon G_1\times G_2\curvearrowright \ku H$ is a continuous affine isometric action with linear part $\rho$, then $\alpha$ almost fixes a point on $\ku H$.
\end{cor}
Shalom's proof of Corollary \ref{thm: shalom} in the locally compact case (see \cite{shalom} p.~ 19--20) uses the existence of contracting projections onto closed convex subsets of a Hilbert space and the analogue of Theorem~\ref{thm: restriction to normal subgroup} in degree $n=1$ for unitary representations of locally compact groups. It applies ad verbatim to our setting, so we need not repeat it here.

U. Bader, A. Furman, T. Gelander and N. Monod \cite{bader} studied the structure of affine actions of product groups on uniformly convex spaces (a subclass of the reflexive spaces) and in this setting obtained a slightly weaker result than Corollary \ref{thm: shalom}. Namely, if  $G_1\times G_2$ is a product of topological groups and $V$ is a uniformly convex $G_1\times G_2$-module
with $V^{G_1}=V^{G_2}=\{0\}$, then either
\begin{enumerate}
\item[(a)] there are almost $G_1\times G_2$-invariant unit vectors in $V$, or
\item[(b)] $H_c^1(G_1\times G_2,V)=\{0\}$.
\end{enumerate}
Theorem \ref{thm: restriction to normal subgroup} is somewhat independent of their statement and shows that
$$
Z^1(G_1\times G_2,V)\subseteq \overline{B^1(G_1,V)}\times \overline{B^1(G_2,V)}
$$
can be added to condition (a) above.

\begin{rem}We also note that the naive generalization of Corollary \ref{thm: shalom} to degrees $n\geqslant 2$ 
is false. If $G=\mathbb F_2\times \mathbb F_2$ and 
$V=\ell^2(G)$, then $\bar H^2(G, V)\ne 0$, because the second $L^2$-Betti number of $G$ is non-zero, but $V^{\mathbb F_2}=\{0\}$ for each of the two copies of $\mathbb F_2$. 
\end{rem}

The following corollary follows from 
Theorem~\ref{thm: restriction to normal subgroup} by taking $N=G$. It generalizes 
earlier results by M.~Puls~\cite{puls}, F.~Martin and A.~Valette~\cite{martin} and 
E.~Kappos~\cite{kappos}. 
Puls and Martin-Valette proved the vanishing of the first $\ell^p$-cohomology ($p\in (1,\infty)$) of a finitely generated group with an element of infinite order in the center or 
infinite center, respectively. P.~Nowak~\cite{nowak} informed us that he independently found a proof of the result 
by Martin-Valette, which is -- like our Theorem~\ref{thm: restriction to normal subgroup} -- based on the Ryll-Nardzewski fixed point theorem. 

The first $\ell^p$-cohomology of a finitely generated group $G$ coincides with 
$\bar H^1(G, \ell^p(G))$ (see~\cite[Lemma 1]{bourdon}). Later, 
Kappos proved that, if the discrete group $G$ satisfies the 
finiteness property FP$_n$ and has infinite center, then 
$\bar H^i(G, \ell^p(G))=0$ for $i\le n$.

\begin{cor}\label{cor: center}
Let $G$ be a topological group and $Z\leqslant G$ its center. Let $V$ be a continuous wap Banach $G$-module such that $V^Z=\{0\}$.
Then for any $n\geqslant 0$
\[
	\redcohom^n(G, V)=0.
\]
\end{cor}

\begin{rem}
	We should mention that Corollary \ref{cor: center} fails for non-wap representations even of abelian groups. To see this, consider $G=\Z^n$ and $V=\ell^1(\Z^n)$. Then Poincare duality implies that
	\[
		H^n\bigl(\Z^n, \ell^1(\Z^n)\bigr)\cong H_0\bigl(\Z^n, \ell^1(\Z^n)\bigr)\cong \ell^1(\Z^n)_{\Z^n}, 
	\]
	where the latter are the co-invariants of $\ell^1(\Z^n)$, i.e., the quotient of 
	$\ell^1(\Z^n)$ by the submodule $W\subset \ell^1(\Z^n)$ generated by $\{gv-v\mid v\in \ell^1(\Z^n), g\in\Z^n\}$. 
	It is not hard to see that the Poincare duality map is a continuous isomorphism, 
	whereby
	\[
		\bar H^n\bigl(\Z^n, \ell^1(\Z^n)\bigr)\cong \bar H_0\bigl(\Z^n, \ell^1(\Z^n)\bigr)\cong \ell^1(\Z^n)/\bar{W}.
	\]
	The last term is not zero, because summation of coefficients 
	induces a non-zero linear functional on $\ell^1(\Z^n)/\bar{W}$.

More concretely,  define a cocycle $b\in Z^1(\Z,\ell^1(\Z))$ by $b(n)=e_0+e_1+\ldots+e_{n-1}$.  Then, if $x=\sum_{n=-k}^ka_ne_n$ is a finitely supported vector, we have
\[\begin{split}
\norm{x-\alpha(1)x}
=&|a_{-k}|+|a_{-k+1}-a_{-k}|+\ldots+|a_{-1}-a_{-2}|+|a_0-a_{-1}+1|\\
&+|a_1-a_0|+\ldots+|a_k-a_{k-1}|+|a_k|\\
\geqslant&1.
\end{split}\]
So $\norm{x-\alpha(1)x}\geqslant 1$, for all $x\in \ell^1(\Z)$, showing that $b\notin \overline{B^1(\Z,\ell^1(\Z))}$.

\end{rem}

As a generalization of Corollary \ref{cor: center} for $n=1$, we have the following result extending Corollary~3.7 in Shalom's paper~\cite{shalom}, which deals with { unitary} 
representations of {locally compact} groups. 

\begin{thm}\label{thm: center}
   Let $G$ be a topological group and $Z\leqslant G$ its center. Let $V$ be a continuous 
   wap Banach $G$-module such that $V^G=\{0\}$. Then the projection 
$G\to G/Z$ induces an isomorphism 
\[
	\redcohom^1(G, V)\cong \redcohom^1(G/Z, V^{Z}). 
\]
\end{thm}

The following result is a generalization of the fact that nilpotent groups have 
property $H_T$~\cite{shalom-harmonic}, i.e., of the fact that nilpotent groups satisfy the conclusion of 
the theorem below for unitary representations. We believe that it holds in each 
degree, but are only able to show it for degree~$1$. In fact, it is true in every degree for simply connected nilpotent Lie groups and unitary coefficients by a result of 
Blanc~\cite{blanc}. 

\begin{thm}\label{thm: nilpotent}
	Let $G$ be a nilpotent topological group. If $V$ is a continuous wap 
	Banach $G$-module such that $V^G=\{0\}$, then $\redcohom^1(G,V)=0$. 
\end{thm}

In contrast to the situation for \emph{reduced} cohomology, every (discrete) infinite 
    amenable group $\Gamma$ has a unitary representation $V$ with 
    $H^1(\Gamma, V)\ne 0$~\cite[Theorem 0.2]{shalom}.

\section{Wap representations: examples and decompositions}\label{sec: wap}
Similarly to weak almost periodicity, a linear representation $\rho\colon G\til GL(V)$ is {\em almost periodic} if the orbits $\rho(G)v$ are relatively norm compact in $V$. Based on the Peter-Weyl Theorem, one can show  (see \cite{shiga} for condition (4)) the equivalence of the following statements,
\begin{enumerate}
\item $\rho\colon G\til GL(V)$ is almost periodic,
\item $\rho(G)$ is relatively sot-compact in $B(V)$,
\item $\rho(G)$ is relatively sot-compact in $GL(V)$,
\item $V$ is the closed linear span of finite-dimensional irreducible subspaces.
\end{enumerate}

We note the difference between the two criteria for almost and weak almost periodicity. For an almost periodic representation $\rho\colon G\til GL(V)$, the sot-closure  of $\rho(G)$ in $B(V)$ is an sot-compact subgroup $\widetilde G$ of $GL(V)$. Moreover, when restricted to bounded subsets of $B(V)$, the composition operation $\big(B(V),{\rm sot}\big)\times \big(B(V),{\rm sot}\big)\til \big(B(V),{\rm sot}\big)$ is jointly continuous in the two variables, which means that $\widetilde G$ is a compact topological group in the strong operator topology.

On the other hand, even when restricted to bounded subsets,  the composition operation  $\big(B(V),{\rm wot}\big)\times \big(B(V),{\rm wot}\big)\til \big(B(V),{\rm wot}\big)$ is, in general, only separately continuous in the two variables.
It follows that, for a weakly almost periodic representation, the wot-closure of $\rho(G)$ in $B(V)$ is a wot-compact subsemigroup $S$ of $B(V)$, but need not be a subgroup. The separate continuity of the multiplication in $S$ is expressed by saying that $S$ is a compact {\em semitopological} semigroup. R. Ellis'  joint continuity Theorem \cite{ellis2} implies that a compact semitopological semigroup, which is algebraically a group, is, in fact, a compact topological group.

The first fact also means that any almost periodic representation of a group $G$ factors through an sot-continuous  representation of a compact group and, conversely, any such representation is obviously almost periodic.

The construction of non-trivial weakly almost periodic representations is somewhat more subtle, since it follows easily from Namioka's joint continuity Theorem \cite{namioka} that a wot-continuous representation of a  compact (in fact, even locally compact or completely metrizable) group is always sot-continuous and hence  also almost periodic. But we have the following three central examples.
\begin{exa} Note that if $V$ is reflexive, then, by the Banach--Alaoglu Theorem, any norm bounded subset of $V$ is relatively weakly compact and thus any bounded linear representation $\rho\colon G\til GL(V)$ of a group $G$ is weakly almost periodic. 
\end{exa}

\begin{exa}Suppose $(X,\mu)$ is a probability space and $G$ is a group acting by measure preserving transformations on $X$. Let also $\rho\colon G\til GL\big( L^p(X,\mu)\big)$ denote the corresponding isometric linear representation of  $G$ on $L^p(X,\mu)$ for $1\leqslant p<\infty$. Then $\rho$ is weakly almost periodic. For $1<p<\infty$, this of course follows from the reflexivity of $L^p(X,\mu)$, while, for $p=1$, this is due to  the fact that the weak operator topologies of $B\big(L^1(X,\mu)\big)$ and $B\big(L^2(X,\mu)\big)$ coincide on the intersection $B\big(L^1(X,\mu)\big)\cap B\big(L^2(X,\mu)\big)$ (Cor. 16.5 \cite{nagel}). 
\end{exa}

\begin{exa}
Let $G$ be a group acting by homeomorphisms on a compact Hausdorff space $K$. Then the corresponding isometric linear representation $\rho\colon G\til GL\big(C(K)\big)$ is weakly almost periodic if and only if all accumulation points of $G$ in $K^K$ are continuous, i.e., whenever a function $f\colon K\til K$ is the pointwise limit of a net in $G$, then $f$ is continuous (Prop.~II.2 \cite{ellis}). This is a consequence of A. Grothendieck's criterium for weak compactness in $C(K)$ \cite{grothendieck}.

On the other hand, the flow $G\curvearrowright K$ is distal if and only if the pointwise closure of $G$ in $K^K$ is a subgroup of $K^K$. It thus follows from Ellis' joint continuity Theorem that, if the flow is both distal and satisfies the criterion for weak almost periodicity, then the pointwise closure of $G$ in $K^K$ is a compact topological group of homeomorphisms in the pointwise convergence topology. By Lebesgue's dominated convergence Theorem, one sees that the representation $\rho\colon G\til GL\big(C(K)\big)$ is wot-continuous, so also sot-continuous and therefore almost periodic.
\end{exa}

Apart from appearing naturally in the contexts above, the weakly almost periodic representations possess another advantage in that they are sufficiently well behaved to induce canonical invariant decompositions of the Banach space $V$. To state these, given a bounded linear representation $\rho\colon G\til GL(V)$, we call a vector $v\in V$
\begin{enumerate}
\item {\em almost periodic} if $\rho(G)v$ is relatively norm compact, and 
\item {\em furtive} if $0\in \overline{\rho(G)v}^{\rm weak}$. 
\end{enumerate}
Note that, since $\rho(G)v$ is bounded away from $0$ for any $v\neq 0$, a non-zero vector $v$ cannot simultaneously be almost periodic and  furtive. 

\begin{thm}[The Jacobs--de Leeuw--Glicksberg decomposition] \label{jdlg}
Every wap Banach $G$-module $V$ admits a decomposition into closed linear $G$-invariant subspaces
\[
V=V_1\oplus V_2,
\]
where $V_1$ is the set of almost periodic vectors and $V_2$ is the set of furtive vectors. 
\end{thm}

The above decomposition is due to K. Jacobs \cite{jacobs} and K. de Leeuw and I. Glicksberg \cite{deleeuw} in slightly more restrictive settings. A full proof of Theorem \ref{jdlg}  based on C. Ryll-Nardzewski's fixed point Theorem \cite{ryll} can be found in \cite{berglund} (see Cor. 6.2.19). 

We now sketch a more direct proof of Theorem \ref{jdlg} due to J.-P. Troallic \cite{troallic}. First, a simple compactness argument shows that any compact semitopological semigroup $S$ contains a unique minimal two-sided ideal $K(S)$ called the {\em Sushkevich kernel} of $S$. Moreover, using I. Namioka's joint continuity Theorem \cite{namioka}, J.-P. Troallic \cite{troallic} gives a simple proof of the following result.
\begin{thm}\label{troallic}
Let $S$ be a compact semitopological semigroup containing  a dense subgroup.
Then $K(S)$ is a compact topological group.
\end{thm}

Now suppose that  $\rho\colon G\til GL(V)$ is a weakly almost periodic representation of a group $G$ and let $S$ denote the closure of $\rho(G)$ in $B(V)$ with respect to the weak operator topology. Then $K(S)$ is a compact topological group and thus, as mentioned above, the tautological representation of $K(S)$ on $V$ is almost periodic. 

Let $e$ denote the identity element in $K(S)$ and note that, being an idempotent operator, $e$ is a linear projection on $V$. We remark that a vector $v\in V$ is furtive if and only if there is some $s\in S$ so that $s(v)=0$. In particular, $V_2=\ker (e)$ is contained in the set of furtive vectors. Conversely, if $v$ is furtive, pick $s\in S$ so that $s(v)=0$ and note that since $es\in K(S)$ there is $t\in K(S)$ so that $e=t\cdot es$, whereby $e(v)=tes(v)=0$, i.e., $v\in V_2$, showing that $V_2$ is the set of furtive vectors.

Suppose, on the other hand, that $v\in V_1=\ker(\Id-e)$. Then $\rho(G)v\subseteq Sv=Se(v)\subseteq K(S)v$, whereby $v$ is $\rho(G)$-almost periodic. Since $V=V_1\oplus V_2$ and $V_1$ is contained in the closed linear subspace $V_{\rm ap}$ consisting of the almost periodic vectors and $V_{\rm ap}$ has trivial intersection with the set $V_2$ of furtive vectors, we must have $V_1=V_{\rm ap}$. Finally, since for all $s\in S$ we have $es=ese=se$, the projection $e$ commutes with $S$ and thus the decomposition is $G$-invariant.

The second decomposition due to L. Alaoglu and G. Birkhoff \cite{alaoglu} splits $V$ into the subspace of invariant vectors and a canonical complement.
\begin{thm}[The Alaoglu--Birkhoff decomposition] \label{ab}
Every wap Banach $G$-module $V$ 
admits a decomposition into closed linear $G$-invariant subspaces
\[
V=V_1\oplus V_2,
\]
where $V_1$ is the set of invariant vectors and $V_2$ is the set of vectors $v\in V$ with $0\in \overline{\rm conv}\big(\rho(G)v\big)$. 
\end{thm}

To see how to obtain the Alaoglu--Birkhoff decomposition from the Jacobs--de Leeuw--Glicksberg decomposition, note first that, by Mazur's Theorem, any furtive vector $v$ satisfies $0\in \overline{\rm conv}\big(\rho(G)v\big)$. So, by restricting the attention to the invariant subspace of almost periodic vectors, one may assume that $\rho\colon G\til GL(V)$ is almost periodic and hence that $\widetilde G=\overline{\rho(G)}^{\rm sot}$ is an sot-compact subgroup of $GL(V)$. Letting $\mu$ be the Haar measure on $\widetilde G$, we can now define a linear projection $P$ of $V$ onto its subspace $V^G$ of $\rho(G)$-invariant vectors via
$$
\psi\big(Px\big)=\int_{\widetilde G} \psi(Tx)\;d\!\mu(T),
$$
for all $\psi\in V^*$ and $x\in V$. Then $V_1=\im P=V^G$ and $V_2=\ker P$.

The third canonical decomposition of interest to us lies between the previous two and constructs an invariant complement to the linear subspace of vectors with finite orbits. We state the result obtained by combining all three decompositions.

\begin{thm}\label{decomp}
Every wap Banach $G$-module $V$ admits a decomposition 
\[
V=V_1\oplus V_2\oplus V_3\oplus V_4
\]
into closed linear $G$-invariant subspaces so that
\begin{enumerate}
\item $V_1$ is the set of $G$-invariant vectors,
\item $V_2$ is the closed linear span of finite-dimensional irreducible subspaces $F$ so that for every such $F$ the $G$-action on $F$ factors through a finite group, 
\item every non-zero vector in $V_3$ is almost periodic with infinite orbit, 
\item $V_4$ is the set of furtive vectors.
\end{enumerate}
\end{thm} 

This last decomposition follows fairly easily from the preceding two. Namely, by the Jacobs--de Leeuw-Glicksberg decomposition, one may assume again that $\rho\colon G\til GL(V)$ is almost periodic and hence that $\widetilde G=\overline{\rho(G)}^{\rm sot}$ is an sot-compact subgroup of $GL(V)$. Let now $L$ denote the connected component of the identity in $\widetilde G$ and let $V=V_1\oplus V_2$ be the Alaoglu--Birkhoff decomposition with respect to the group $L$. Since $V_1$ is the subspace of $L$-invariant vectors, the action of $\widetilde G$ on $V_1$ factors through the profinite group $\widetilde G/L$. Thus, $V_1$ is the closed linear span of finite-dimensional irreducible subspaces $F$ and, on every such $F$, the action of $\widetilde G/N$, and hence also of $G$, factors through a finite quotient.  On the other hand, on $V_2$, the compact connected group $L$ acts non-trivially and hence has infinite orbits.

\section{A lemma about convex averages}
We recall C. Ryll-Nardzewski's fixed point Theorem \cite{ryll}, which is of central importance in the following.
A simple geometric proof can be found in \cite{asplund, fabian}.

\begin{thm}[Ryll-Nardzewski]\label{thm: ryll-nardzewski}
Let $G$ be a group acting by affine isometries on a weakly compact convex subset $C$ of a Banach space $V$. Then there is a vector $v\in C$ fixed by $G$.
\end{thm}

The $G$-action on a Banach $G$-module $V$ extends to an action of the 
group algebra $\R G$ by linearity. In the group algebra, 
we denote by $\Delta_G$ the simplex generated by the elements of $G$,
that is, 
\[ 
\Delta_G= \bigl\{\sum_{i=1}^n \lambda_i h_i \in \R G \del n \in \N,~\lambda_i\in [0,1],
~h_i\in G,~\sum_{i=1}^n \lambda_i =1\bigr\}. 
\]

\begin{lem} \label{compact-RN} 
Let $V$ be a wap Banach $G$-module and $K\subseteq V$ a norm compact subset.  Then for every $\epsilon>0$ there is $\delta \in\Delta_G$
so that ${\rm dist}(\delta\cdot v, V^G) <\epsilon$ for every $v\in K$. 
\end{lem}

\begin{proof}
Since a wap representation is automatically bounded, we may and will, upon replacing 
the norm on $V$ by an equivalent one, assume that 
$G$ acts by isometries on $V$. So every element of $\Delta_G$ acts as a contraction.
So let $\{v_1,\ldots, v_n\}\subseteq K$ be a finite $\epsilon/2$-net in $K$ and note that it suffices to find $\delta\in \Delta_G$ so that ${\rm dist}(\delta\cdot v_i, V^G) <\epsilon/2$ for every $i$.

Let now $U=\big(\bigoplus_{i=1}^n V\big)_{\ell_\infty}$ denote the $\ell_\infty$ sum of $n$ copies of $V$ and set $x=(v_1,\ldots,v_n)$. Note that with the diagonal $G$-action, $U$ becomes a wap $G$-module. Since $G\cdot x$ is relatively weakly compact in $U$, by Krein's Theorem (Thm.\! 3.133 \cite{fabian}), 
$$
C={\overline{\rm conv}}^{\norm\cdot}\big(G\cdot x\big)={\overline{\big(\Delta_G\cdot x\big)}}^{\norm\cdot}
$$
is a weakly compact convex subset of $U$ invariant under the $G$-action. Since $G$ acts by affine isometries on $C$, by the Ryll-Nardzewski fixed point Theorem, there is some $y\in C$ fixed by $G$, i.e., $y\in C\cap U^G$. 
Pick $\delta\in \Delta_G$ so that $\norm{\delta\cdot x-y}<\eps/2$, whereby ${\rm dist}(\delta\cdot v_i, V^G) <\epsilon/2$ for all $i$ and thus ${\rm dist}(\delta\cdot v, V^G) <\epsilon$ for all $v\in K$.
\end{proof}

%
%

\section{Continuous cohomology of topological groups} 
\label{sec:cohomological_preliminaries}

We review some basic facts about continuous cohomology of topological groups. 
In the sequel, let $G$ be a topological group and $V$ be a continuous Banach $G$-module. 
Let $C(G^n, V)$ be the vector space of maps $G^n\to V$ that are continuous 
with respect to the norm-topology on $V$; we endow this space with 
the compact-open topology. The group $G$ acts continuously from the left on 
$C(G^n, V)$ by 
\[(g\cdot c)(g_0,\ldots,g_{n-1})=gc(g^{-1}g_0,\ldots, g^{-1}g_{n-1}).\] 
The 
invariants $C(G^n, V)^G$ is the closed subspace of equivariant maps in 
$C(G^n,V)$. 

Recall that the \emph{standard homogeneous resolution} is the complex  $C(G^{\ast+1},V)$
\[
	C(G, V)\xrightarrow{\partial} C(G^2, V)\xrightarrow{\partial} \ldots\xrightarrow{\partial} C(G^{n+1}, V)\to\ldots
\]
starting in degree zero with differentials\footnote{As it is common, the degree is dropped from the notation of a differential.}, which are continuous and $G$-equivariant,  given by 
\[
	(\partial c)(g_0,\ldots, g_n)=\sum_{k=0}^{n}(-1)^kc(g_0,\ldots, \hat{g_i}, \ldots, g_n). 
\]
The $n$-th \emph{continuous cohomology} $\cohom^n(G;V)$ of $G$ with coefficients in $V$ is commonly defined 
as the $n$-th cohomology of the complex $C(G^{\ast+1},V)^G$. The $n$-th \emph{reduced continuous cohomology} $\redcohom^n(G;V)$ is defined by the quotient of the space of cocycles \[Z^n(G;V):=\ker(\partial)\subset C(G^{n+1},V)^G\] 
by the closure of the space of coboundaries 
\[B^n(G;V):=\im(\partial)\subset C(G^{n+1},V)^G.\] 


Let us recall the functoriality 
of continuous group cohomology as a functor in two variables. 
Let $\scrG$ be the category of pairs $(G, V)$, 
where $G$ is a topological group, $V$ is a Banach $G$-module, and a morphism from 
$(G,V)$ to $(G',V')$ is a pair 
$(\alpha: G\to G', f: V'\to V)$ consisting of a continuous group 
homomorphism $\alpha$ and a bounded linear operator $f$ 
such that $f(\alpha(g)v)=gf(v)$ for every $v\in V'$ and $g\in G$.
Such a morphism $(\alpha, f)$ induces a continuous chain map 
\begin{align*}
	C(\alpha, f)\colon& C(G'^{n+1}, V')^{G'}\to C(G^{n+1},V)^G\\&c\mapsto \bigl((g_0,\ldots, g_n)\mapsto f(c(\alpha(g_0),\ldots,\alpha(g_n)))\bigr)
\end{align*}
and thus maps in (reduced) continuous cohomology which we denote by $\cohom^\ast(\alpha,f)$ and 
$\redcohom^\ast(\alpha, f)$, respectively. Both continuous cohomology and 
reduced continuous cohomology are functors from $\scrG$ to abelian groups. 

Suppose $g\in G$. Let $\rho(g)\colon V\to V$ be multiplication by $g$ and 
$conj_g:G\to G$, $conj_g(h)=g^{-1}hg$ be conjugation by $g$. Then $(conj_g, \rho(g))$ is 
a $\scrG$-morphism from $(G,V)$ to itself, and, for $c\in C(G^{n+1},V)^G$, $C(conj_g, \rho(g))(c)$ is given by
\[
C(conj_g, \rho(g))(c)(h_0, \ldots, h_{n})=gc(g^{-1}h_0g,\ldots,g^{-1}h_0g)=
c(h_0g,\ldots, h_{n}g).
\]
It is well known that

\begin{lem}\label{lem: identity}
	$\cohom^\ast(conj_g, \rho(g))$ and $\redcohom^\ast(conj_g, \rho(g))$ are the identities 
	on $\cohom^\ast(G;V)$ and $\redcohom^\ast(G;V)$, respectively.
\end{lem}

See~\cite[(8.3) Proposition on p.~80]{brown-book} for a homological-algebra proof in the setting of discrete groups. To see that it still holds in the setting of 
continuous cohomology for arbitrary topological groups, we give an explicit and 
continuous chain homotopy between $\id_{C(G^{\ast+1}, V)^G}$ and $\phi:= C(conj_g,\rho(g))$. 

\begin{proof}
	It is straightforward to verify that the collection of homomorphisms for $n\geqslant 0$
	\begin{gather*}
		h: C(G^{n+1}, V)^G\to C(G^n, V)^G\\
		c\mapsto \bigl((g_0,\ldots, g_{n-1})\mapsto \sum_{k=0}^{n-1}(-1)^{k+1}c(g_0,\ldots, g_k, g_k g,\ldots, g_{n-1}g)\bigr)
	\end{gather*}
    satisfies $h\partial+\partial h=\id-\phi$, 
so provides a chain homotopy between $\id$ and $\phi$, which is obviously continuous.  
\end{proof}

\section{Proofs of cohomological statements} 
\label{sec:proofs_of_cohomological_statements}


\begin{proof}[Proof of Theorem~\ref{thm: restriction to normal subgroup}]
Denote the representation by $\rho\colon G\to GL(V)$ and the inclusion by $j\colon N\to G$.  
For $h\in C$ the following diagram of morphisms in $\scrG$ is commutative 
since $h$ centralizes $N$: 
\[\xymatrixcolsep{5pc}
	\xymatrix{ (N, V)\ar[d]_{(\id, \rho(h))}\ar[r]^{(j, \id_V)} &(G, V)\ar[d]^{(conj_h, \rho(h))}\\
	           (N, V)\ar[r]^{(j, \id_V)} & (G, V)
	}
\]
Hence we obtain a commutative diagram of abelian groups 
\[\xymatrixcolsep{5pc}
	\xymatrix{
	\cohom^n(G, V)\ar[d]_{\cohom^n(conj_h, \rho(h))}\ar[r]^{\cohom^n(j,\id_V)}& \cohom^n(N,V)\ar[d]^{\cohom^n(\id, \rho(h))}\\
	\cohom^n(G,V)\ar[r]^-{\cohom^n(j,\id_V)} & \cohom^n(N, V)
	}
\]
The group $C$ acts on $C(N^{n+1}, V)^N$ and $\cohom^n(N, V)$ by 
$C(\id, \rho(h))$ and $\cohom^n(\id, \rho(h))$, respectively, for $h\in C$. By linearity these actions extend to 
endomorphisms $\rho(\delta)$ for every element $\delta\in \Delta_C$ in 
the standard simplex of the group ring of $C$.  
By Lemma~\ref{lem: identity} 
the left vertical map $\cohom^n(conj_h, \rho(h))$ is the identity. Hence 
$\rho(\delta)(x)=x$ for every $\delta\in \Delta_C$ and every $x$ in the image 
of the restriction map $\cohom^n(j,\id_V)$. 
Let $c\in C(N^{n+1}, V)^N$ be a cocycle that lies in the image of $C(j, \id_V)$. 
Let $\epsilon>0$ and a compact subset $K\subset N^{n+1}$ be given. 
By applying Lemma~\ref{compact-RN} to the compact subset $c(K)\subset V$, we find 
some $\delta\in\Delta_{C}$ with 
\[
	\sup_{z\in K}\norm{\rho(\delta)(c)(z)}<\epsilon. 
\]
By the discussion above there is some coboundary $b$ with 
$c=\rho(\delta)(c)+b$. Since $\epsilon$ and $K$ were arbitrary, it follows 
that the class of $c$ is zero in reduced continuous cohomology. 
\end{proof}

\begin{proof}[Proofs of Theorems~\ref{thm: center} and~\ref{thm: nilpotent}]
	By the Alaoglu-Birkhoff decomposition (Theorem~\ref{ab}) 
	there is a $Z$-invariant decomposition 
	\[V=V^Z\oplus V_0\] 
	with $V_0=\{v\in V\mid 0\in \overline{\rm conv}(\rho(G)v)\bigr\}$. 
	Since $Z$ is central, this decomposition is also $G$-invariant. 
	Theorem~\ref{thm: restriction to normal subgroup} applied to $N=G$ yields 
	$\redcohom^\ast(G, V_0)=0$, thus 
	\begin{equation}\label{eq: reduction}
	\redcohom^\ast(G, V)\cong \redcohom(G, V^Z).
	\end{equation}
	Now assume that $V^G=\{0\}$. Under this assumption we prove that the 
	projection $G\to G/Z$ induces a commutative diagram 
	\[
		\xymatrix{
              C(G/Z, V^Z)^{G/Z}\ar[d]^\partial\ar[r]^-\cong & C(G, V^Z)^G\ar[d]^\partial\\
              Z^1(G/Z, V^Z)\ar[r]^-\cong & Z^1(G, V^Z)
		}
	\]	
	with the horizontal arrows being topological isomorphisms (this obvious for the 
	upper one). This will yield $\redcohom^1(G,V^Z)\cong \redcohom^1(G/Z, V^Z)$ 
	and by~\eqref{eq: reduction} it will imply 
	Theorem~\ref{thm: center}. To see that the lower horizontal arrow is a 
	topological isomorphism it suffices to show that $c\in Z^1(G, V^Z)$ 
	satisfies the relation 
	\begin{equation}\label{cocycle equation}
		c(hg_1,g_2)=c(g_1,hg_2)=c(g_1,g_2)\text{ for all $h\in Z$ and $g_1,g_2\in G$.}
	\end{equation}
    For fixed $h\in Z$ and arbitrary $g\in G$ we have 
	$0=\partial c(h,g,hg)-\partial c(e,h,g)=gc(e,h)-c(e,h)$ yielding $c(e,h)=0$ 
	because of $V^G=\{0\}$. 
	Hence $0=\partial c(h,g,hg)=c(g,hg)-c(h,hg)+c(h,g)=gc(e,h)-c(e,g)+c(h,g)=-c(e,g)+c(h,g)$ from 
	which~\eqref{cocycle equation} directly follows by $G$-equivariance. 
	
	To prove Theorem~\ref{thm: nilpotent} first note that it holds for abelian groups 
by Theorem~\ref{thm: restriction to normal subgroup}. With Theorem~\ref{thm: center} 
we can now run an obvious induction over the nilpotency degree. 
\end{proof}

\end{document}